\documentclass[11pt, oneside]{amsart}
\usepackage{geometry}                
\usepackage{graphicx}
\usepackage{amssymb}
\usepackage{epstopdf}
\usepackage{amsmath}
\usepackage{mathrsfs}
\title[Sufficient conditions for local solvability of some degenerate PDO with ...]{Sufficient conditions for local solvability of some degenerate PDO with complex subprincipal symbol}
\author{Serena Federico}
\email{serena.federico2@unibo.it}
\address{University of Bologna, Piazza di Porta San Donato 5, Bologna, Italy.}

\newtheorem{theorem}{Theorem}

\newtheorem{definition}[]{Definition}
\newtheorem{remark}[]{Remark}
\newtheorem{example}[]{Example}

\newcommand{\n}{\| }

\begin{document}
\maketitle

\begin{abstract}
We will show a local solvability result for a class of degenerate second order linear partial differential operators with a complex subprincipal symbol. Due to the form of the operators in the class the subprincipal symbol is invariantly defined and we shall give sufficient conditions for the local solvability to hold involving the real and the imaginary part of the latter. Under suitable conditions we will prove that the class under consideration is $L^2$ to $L^2$ locally solvable.
\end{abstract}

\section{INTRODUCTION}
We consider the local solvability problem for a class of degenerate linear partial differential operators of the form
\begin{equation}
\label{P}
P=\sum_{j=1}^NX_j^*f_jX_j+iX_0+X_{N+1}+a_0,
\end{equation}
where the $X_j=X_j(x,D)$, $0\leq j\leq N+1$, $D=(D_1,...,D_n)=-i(\partial_{x_1},...,\partial_{x_n})$, are homogeneous first order linear partial differential operators with smooth coefficients defined on an open set $\Omega\subseteq\mathbb{R}^n$, $f_j=f_j(x)\in C^\infty(\Omega;\mathbb{R})$ are real smooth functions possibly vanishing at some point of the set $\Omega$ and $a_0\in C^\infty(\Omega;\mathbb{C})$. In addition we consider the operators $X_j$, $0\leq j\leq N+1$, having real coefficients (i.e the $iX_j$, $0\leq j\leq N+1$, are real smooth vector fields),. 

Our goal here is to give sufficient conditions for the local solvability of $P$ to hold at each point of $\Omega$. Of course, the most interesting cases are covered by situations in which we have local solvability at points where the operator $P$ is degenerate. Moreover, we will not suppose that the vector fields (i.e. the operators) $iX_j$, $1\leq j\leq N+1$, are nondegenerate, therefore very complicated situations can arise in the class \eqref{P} above.

The sufficient conditions we are going to introduce will be given in terms of the operators $X_j$ and of the functions $f_j$, and, as we shall see, they will be essentially requirements on the subprincipal part of the operator, which are therefore invariantly defined. 

The study of operators such as \eqref{P} goes back to the work of Kannai \cite{Ka} where an example of hypoelliptic non-solvable operator is given. To be precise, the so called Kannai operator, is not locally solvable around points where the principal symbol changes sign, showing that this property is meaningful for the local solvability.
Some generalizations of this operator have been studied by Colombini, Cordaro and Pernazza in \cite{CCP} and by Colombini, Pernazza and Treves in \cite{CCP1}, and further extensions have been given by the author and A. Parmeggiani in \cite{FP} and \cite{FP1} (see also \cite{Par}).
The form of the operator $P$ in \eqref{P} is mostly linked to that in \cite{FP} and \cite{FP1} and our aim is to cover other open cases. With respect to the class in \cite{FP1}, here we allow the presence of several functions in the second order part of the operator even when $X_{N+1}(x,D)\equiv 0$. Moreover, while in \cite{FP} and \cite{FP1} the condition $iX_0f>0$ near $f^{-1}(0)$, with $f_j= f$ for all $j=1,...,N$, is required, here we allow $iX_0f_j\geq 0$ over $\Omega$, for all $j=1,...,N$. Let us also remark that for \eqref{P} we require $X_0(x,D)\not\equiv 0$ and $X_{N+1}$ possibly such that $X_{N+1}\equiv 0$ or degenerate. The cases $X_0\not\equiv 0$ in presence of a single function $f=f_j$ for all $j$, and the case $X_0\equiv 0, X_{N+1}\not\equiv 0$ with several functions $f_j$ are studied in \cite{FP1}. 

There are other interesting results about operators with double characteristics which are somehow connected to operators in the class \eqref{P}. We recall, for instance, the results of Helffer in \cite{Hel} where he shows the construction of parametrices for some operators with double characteristics and gives, as an application, examples of operators which are also contained in the class \eqref{P}.
We have results in \cite{T} due to Treves where operators of the form $XY+Z$ are considered. Very interesting results for pseudodifferential operators appear in works by Mendoza \cite{M} and Mendoza and Uhlmann \cite{MU, MU2} in which necessary and sufficient conditions for the local solvability of operators having a principal symbol with specific form are given. 
In \cite{KP} Pravda-Starov exhibits some examples of weakly hyperbolic operators which are not locally solvable consistently with the results of Mendoza and Uhlmann.
In this regard we want to highlight that, due to the required conditions, all the partial differential operators contained in the class treated in \cite{MU,MU2} do not satisfy condition $\mathsf{Sub}(\mathscr{P})$ of \cite{MU,MU2} except for dimension 2.  In this case, that is when $n=2$, we can find a deep connection between the sufficient conditions given here for \eqref{P} and those given in \cite{MU2}.

It is also worth to recall recent results due to Dencker in \cite{D} about necessary conditions for the local solvability of pseudodifferential operators of subprincipal type (that essentially have involutive characteristic set).

Finally, other results connected to the argument can be found in \cite{F}, \cite{Mu}, \cite{MR1}, \cite{PP} and \cite{PR}.

The present paper is organized as follows.

In Section 2 we will present the hypotheses that determine the class, recall the general definition of local solvability and prove the solvability result by means of a priori estimate. The key point in the proof will be the use of the Fefferman-Phong inequality on a new operator $P'$ appearing in the estimate.

In Section 3 we will give a few examples of operators in the class whose local solvability is guaranteed by the theorem of Section 2. 

\section{STATEMENT AND PROOF OF THE RESULT}
Recall that we are dealing with the local solvability problem for the class of operators of the form \eqref{P} given in the Introduction.

In what follows we shall assume that conditions (H1), (H2) and (H3) below are satisfied:
\begin{itemize}
  \item [(H1)]  $X_0(x,D)\neq 0$ $\forall x\in\Omega$ and $iX_0f_j\ge0$ on $\Omega$ for all $1\leq j\leq N$;
  \item [(H2)] $[X_0,X_j](x,D)=0$ in $\Omega$ for all $1\leq j\leq N$;
  \item [(H3)] for all $x_0\in\Omega$ there exists $U\subset\Omega$ open and bounded containing $x_0$, and  a positive constant $C$ such that 
   $$|\{X_0,X_{N+1}\}(x,\xi)|^2\leq C\left(\sum_{j=1}^N \bigl(iX_0f_j(x)\bigl)X_j(x,\xi)^2+ X_0(x,\xi)^2\right),\quad \forall (x,\xi)\in U\times \mathbb{R}^n;$$
\end{itemize}
where we denote by $X_j(x,\xi)$ the (total) symbol of the operator $X_j$ and by $\{\cdot,\cdot\}$ the Poisson bracket.
 
\begin{remark}
Observe that the subprincipal symbol of $P$ in \eqref{P} is given by $Sub(P)(x,\xi)=iX_0(x,\xi)+X_{N+1}(x,\xi)$, thus it is invariant on $T^*\Omega$ and not only on the double characteristics set since it is the principal symbol of a first order operator.  Moreover one gets that (H1), (H2) and (H3) are essentially requirements on the real and the imaginary part of $Sub(P)$. In particular the imaginary part has a key role here. As we shall see below, conditions (H1) and (H2) allow a control on the commutator (or the Poisson bracket at the level of symbols) between the principal part and the imaginary subprincipal part, while condition (H3) imposes a relation between the real and the imaginary part of $Sub(P)$. 
Since the operator $P$ is degenerate and has a principal symbol which may change sign, the only chance is to control it by means of the first order part. Here the control is guaranteed by the hypothesis (H1) on $\mathsf{Im}Sub(P)=X_0$. In fact the latter condition gives the validity of a Poincar\'e type inequality for $X_0$ that will be used to absorb the $L^2$-errors coming from the principal part and the term $X_{N+1}$.
\end{remark}

The local solvability result we are going to show is proved by means of a priori estimates. 
Before giving the statement of the result we recall below a general definition of $H^s$ to $H^{s'}$ local solvability for a partial differential operator, where $H^s$ stands for the Sobolev space of order $s$.

\begin{definition}
Let P be a partial differential operator defined on an open set $\Omega \subseteq\mathbb{R}^n$. We say that P is $H^s$ to $H^{s'}$ locally solvable at $x_0\in\Omega$ if there exists a compact $K\subset \Omega$, with $x_0\in \mathring{K}=U$ (where $ \mathring{K}$ is the interior of $K$), such that for all $f\in H^s_{\mathrm{loc}}(\Omega)$ there is $u\in H^{s'}_{\mathrm{loc}}(\Omega)$ which solves $Pu=f$ in $U$. 
\end{definition}
\begin{theorem}\label{Thm}
Let $P$ be the operator in \eqref{P} satisfying conditions (H1), (H2) and (H3). Then, for all $x_0\in\Omega$, $P$ is $L^2$ to $L^2$ locally solvable at $x_0$.
\end{theorem}

In order to get the result in Theorem \ref{Thm} it suffices to prove the following a priori estimate that we shall call \textit{solvability estimate}: \textit{For all $x_0\in \Omega$ there exists a compact set $K\subset\Omega$ containing $x_0$ in its interior $\mathring{K}=U$ and a positive constant $C$ such that, for all $\varphi\in C_0^\infty(U)$},
\begin{equation}\label{solvest}
\n P^* \varphi \n\geq C\n \varphi\n,
\end{equation}
\textit{with $\n \cdot \n$ denoting the $L^2$-norm and $P^*$ the adjoint of $P$.}

In view of the well-known equivalence between $H^s$ to $H^{s'}$ local solvability and validity of suitable a priori estimates (see, for instance, \cite{L}), the proof of the theorem is mainly concerned in obtaining the inequality above for the operator $P^*$, giving as a consequence the local solvability of $P$ in the sense $L^2$ to $L^2$.
\begin{proof}[Proof of Theorem \ref{Thm}]
First note that, since $X_0$ is nondegenerate in $\Omega$, we can always find a change of coordinates such that $X_0(x,D)=D_1$. Note also that conditions (H1), (H2) and (H3) are still satisfied, since they are invariant under changes of coordinates. Therefore let us assume $X_0(x,D)=D_1$, and observe that $X_0^*=X_0$ and $X_{j}^*=X_{j}+d_{j}$, where $d_{j}=-i\mathsf{div}(iX_j)$.

We now pick an arbitrary point $x_0\in \Omega$ and start the proof of the solvability inequality by estimating the term
\begin{equation}
\begin{gathered}
2\mathsf{Re}(P^*\varphi, -i\underbrace{X_0^*}_{=X_0}\varphi)= 2\sum_{j=1}^N\mathsf{Re}(X_j^*f_jX_j\varphi, -iX_0^*\varphi)+2\mathsf{Re}(-iX_0^*\varphi, -iX_0^*\varphi)\\
+2\mathsf{Re}(X_{N+1}\varphi, -iX_0^*\varphi)+2\mathsf{Re}(\overline{a_0}\,\varphi, -iX_0^*\varphi)\\
\geq\underbrace{\sum_{j=1}^N2\mathsf{Re}(X_j^*f_jX_j\varphi, -iX_0^*\varphi)}_{(\ref{est1}.1)}+2\n X_0\varphi\n^2\\
\label{est1}
+\underbrace{2\mathsf{Re}(X_{N+1}\varphi, -iX_0\varphi)}_{(\ref{est1}.2)}-\frac {1}{\delta_0}\n\overline{a_0}\n^2_{L^\infty(K)}\n\varphi\n^2-\delta_0\n X_0\varphi\n^2,
\end{gathered}
\end{equation}
for all $\varphi \in C_0^\infty(K)$, where $K$ is a compact set in $\Omega$ containing $x_0$ in its interior and $\delta_0$ is a positive constant that will be chosen later. We then consider the terms $(\ref{est1}.1)$ and $(\ref{est1}.2)$ separately.

For the term $(\ref{est1}.1)$ we have that, for all $\varphi\in C_0^\infty(K)$,
\begin{equation} \label{term1}
\begin{gathered}
\sum_{j=1}^N2\mathsf{Re}(X_j^*f_jX_j\varphi, -iX_0\varphi)=\sum_{j=1}^N\left[(X_j^*f_jX_j\varphi, -iX_0\varphi)+(-iX_0\varphi,X_j^*f_jX_j\varphi)\right]\\
=\sum_{j=1}^N\left[(iX_0X_j^*f_jX_j\varphi, \varphi)+(-iX_0\varphi,X_j^*f_jX_j\varphi)\right]\\
=\sum_{j=1}^N([iX_0,X_j^*f_jX_j]\varphi,\varphi).
\end{gathered}
\end{equation}

For the term $(\ref{est1}.2)$ we have, for all $\varphi\in C_0^\infty(K)$,
\begin{equation*}
\begin{gathered}
2\mathsf{Re}(X_{N+1}\varphi, -iX_0\varphi)=(X_{N+1}\varphi, -iX_0\varphi)+(-iX_0\varphi,X_{N+1}\varphi)\\
=([iX_0,X_{N+1}]\varphi,\varphi)+(iX_0\varphi,X_{N+1}^*\varphi)-(iX_0\varphi,X_{N+1}\varphi)\\
=([iX_0,X_{N+1}]\varphi,\varphi)+(iX_0\varphi,d_{N+1}\varphi),
\end{gathered}
\end{equation*}
whence,
\begin{equation}\label{term2}
\begin{gathered}
2\mathsf{Re}(X_{N+1}\varphi, -iX_0\varphi)
=\mathsf{Re}\bigl(([iX_0,X_{N+1}]\varphi,\varphi)\bigl)+\mathsf{Re}\bigl((iX_0\varphi,d_{N+1}\varphi)\bigl)\\
\geq -\delta_1\n[X_0,X_{N+1}]\varphi\n^2-\frac{1}{\delta_1}\n\varphi\n^2 -\delta_2\n X_0\varphi\n^2-\frac{1}{\delta_2}\n d_{N+1}\n^2_{L^\infty(K)}\n\varphi\n^2,
\end{gathered}
\end{equation}
where $\delta_1$ and $\delta_2$ are two positive constants that will be chosen later.

Therefore, by \eqref{term1} and \eqref{term2}, we get
\begin{equation}
\label{est2}
\begin{gathered}
2\mathsf{Re}(P^*\varphi, -iX_0^*\varphi)\geq \sum_{j=1}^N([iX_0,X_j^*f_jX_j]\varphi,\varphi)+\n X_0\varphi\n^2\\
-\delta_1\n[X_0,X_{N+1}]\varphi\n^2+(1-\delta_0-\delta_2)\n X_0\varphi\n^2-\left[\frac{1}{\delta_2}\n d_{N+1}\n^2_{L^\infty(K)}+\frac {1}{\delta_0}\n\overline{a_0}\n^2_{L^\infty(K)}+\frac{1}{\delta_1}\right] \n\varphi\n^2.
\end{gathered}
\end{equation}
We now write $\n X_0\varphi\n^2=(X_0 \varphi,X_0\varphi)=(X_0^2\varphi,\varphi)$ and, similarily, $\n[X_0,X_{N+1}]\varphi\n^2=([X_0,X_{N+1}]^*[X_0,X_{N+1}]\varphi,\varphi)$, so that the first three terms on the righthand side of \eqref{est2} can be written as
\begin{equation}
\label{P'}
\Big((\sum_{j=1}^N[iX_0,X_j^*f_jX_j]+X_0^2-\delta_1[X_0,X_{N+1}]^*[X_0,X_{N+1}])\varphi,\varphi\Big):=(P'\varphi,\varphi),
\end{equation}
and 
\begin{equation}
\label{ReP}
\begin{gathered}
2\mathsf{Re}(P^*\varphi, -i\underbrace{X_0}_{=X_0^*}\varphi)\geq(P'\varphi,\varphi)+(1-\delta_0-\delta_2)\n X_0\varphi\n^2\\
-\left[\frac{1}{\delta_2}\n d_{N+1}\n^2_{L^\infty(K)}+\frac {1}{\delta_0}\n\overline{a_0}\n^2_{L^\infty(K)}+\frac{1}{\delta_1}\right] \n\varphi\n^2.
\end{gathered}
\end{equation}
Our next goal is to prove that the operator $P'$ given in \eqref{P'}, that is $P'=\sum_{j=1}^N[iX_0,X_j^*f_jX_j]+X_0^2-\delta_1[X_0,X_{N+1}]^*[X_0,X_{N+1}]$, satisfies the Fefferman-Phong inequality in a compact set of $\Omega$ containing $x_0$ in its interior. This will give a control on the term \eqref{P'}, a control being necessary in order to get the solvability estimate \eqref{solvest}. 

To prove the inequality for $P'$ we shall proceed as follows: we will define an operator $A$ which extends $P'$ globally and prove the Fefferman-Phong inequality for the latter. As a consequence we will get the same result for $P'$ in a suitable compact set containing $x_0$ in its interior.

Let us consider a sequence of compact sets $K_0\subseteq K_0' \Subset K_1\subseteq K_1'\subset \Omega$ such that $K_0$ contains $x_0$ in its interior and condition (H3) is satisfied in $K_1'$. Let $\chi_0$ and $\chi_1$ be two functions such that $\chi_\ell\in C_0^\infty(K_\ell')$, $\chi_\ell\equiv 1$ in $K_\ell$ and $0\leq \chi_\ell\leq 1$ in $K_\ell'$ for $\ell=0,1$, and take the operators $\tilde{Y}$, $\tilde{X}_0$, $\tilde{X}_j$, $1\leq j\leq N$, of the form $\tilde{Y}(x,D)=\chi_1(x)[X_0,X_{N+1}](x,D)$, $\tilde{X_0}(x,D)=\chi_1(x)X_0(x,D)$ and $\tilde{X}_j(x,D)=\chi_1(x)X_j(x,D)$ for $1\leq j\leq N$ respectively. We now define the operator $A$ as the Weyl quantization of the symbol $a(x,\xi)\in S^2(\mathbb{R}^n\times\mathbb{R}^n)$ given by

$$ a(x,\xi)=\chi_0\Big(\sum_{j=1}^N\frac1 i\{ip_0,\overline{p}_j\#f_j\#p_j\}+p_0\#p_0-\delta_1\overline{q}\#q\Big),$$
with $q(x,\xi)$ and $p_j(x,\xi)$, $0\leq j\leq N$, denoting the Weyl symbols of the operators $\tilde{Y}$ and $\tilde{X}_j$ respectively, and compute the symbol $a$ (which is such that $a|_{\pi^{-1}(K_0)}=p'|_{\pi^{-1}(K_0)}$ with $p'$ Weyl symbol of $P'$) by means of the Weyl calculus of pseudo-differential operators.

Since $p_j(x,\xi)=p^1_j(x,\xi)+ip_j^0(x,\xi)$, with $p_j^1(x,\xi)=\chi_1(x)X_j(x,\xi)\in S^1(\mathbb{R}^n\times\mathbb{R}^n)$ and $p_j^0(x,\xi)=p_j^0(x)\in S^0(\mathbb{R}^n\times\mathbb{R}^n)$ (with $p_0^0(x)=0$), we have 
\begin{gather*}
\overline{p}_j\#f_j\#p_j=\overline{p}_j\#(f_jp_j+\frac {1}{2 i} \{f_j,p_j\})\\
=f_j\overline{p}_jp_j+\frac {1}{2 i} \{\overline{p}_j,f_jp_j \}+\frac {1}{2 i} \overline{p}_j\{f_j,p_j\}+\frac{1}{2 i}\big\{ \overline{p}_j,\frac {1}{2 i} \{f_j,p_j\} \big\}\\
=f_j|p_j|^2+\frac {1}{2 i} \{\overline{p}_j,f_j\}p_j +\frac {1}{2 i} f_j\{\overline{p}_j,p_j \}+\frac {1}{2 i} \overline{p}_j\{f_j,p_j\}+\frac {1}{2 i}\big\{ \overline{p}_j,\frac {1}{2 i} \{f_j,p_j\} \big\},\\
=f_j|p_j|^2+\frac {1}{2 i} \{\overline{p}_j,f_j\}p_j +\frac {1}{2 i} \overline{p}_j\{f_j,p_j\}+r_0\\
=f_j|p_j|^2+\frac {1}{2 i} \{p^1_j-ip_j^0,f_j\}(p^1_j+ip_j^0) +\frac {1}{2 i} (p^1_j-ip_j^0)\{f_j,p^1_j+ip_j^0\}+r_0\\
\underset{(\{f_j,p_j^0\}=0)}{=}f_j((p^1_j)^2+(p^0_j)^2)+r_0\\
=f_j({p^1_j})^2+r_0,
\end{gather*}
where we denoted by $r_0=r_0(x)$ a (new) smooth compactly supported function with support in $\Omega$.
We then have
\begin{gather*}
\chi_0\sum_{j=1}^N\frac1 i\{ip_0,\overline{p}_j\#f_j\#p_j\}=\chi_0\sum_{j=1}^N\{p_0,f_j(p^1_j)^2\}+\chi_0r_0\\
=\chi_0\sum_{j=1}^N \big(\{p_0,f_j\}(p^1_j)^2+\underbrace{\chi_0f_j\{p_0,(p^1_j)^2\}}_{=0}+\chi_0r_0,
\end{gather*}
where $\chi_0\,f_j\{p_0,(p^1_j)^2\}=0$ since (by condition (H2), that is $\{X_0,X_j\}(x,\xi)=0$) we have
$$\chi_0\,f_j\{p_0,(p^1_j)^2\}(x,\xi)= 2\chi_0\sum_{j=1}^N f_j \Big( \chi_1 X_j\{X_0,\chi_1\}+\chi_1 X_0\{\chi_1,X_j\}\Big)\chi_1X_j,$$
(with $X_j=X_j(x,\xi)$ symbols of $X_j(x,D)$), and therefore, since $\mathrm{supp}\,\chi_0\bigcap \mathrm{supp} \{X_j,\chi_1\}=\emptyset$, $j=0,1,$ we get that the quantity above is identically zero.

Then, as $p_0\#p_0=p_0^2$ and $\overline{q}\#q=q^2+r_0$, we have that
\begin{equation}
\label{FP}
a(x,\xi)=\chi_0\big(\sum_{j=1}^N \{p_0,f_j\}(p_j^1)^2+p_0^2-\delta_1\,q^2+r_0\big)\end{equation}
$$=\chi_0(x)^2\left(\sum_{j=1}^N \{X_0,f_j\}(x)X_j(x,\xi)^2+X_0(x,\xi)^2-\delta_1\{X_0,X_{N+1}\}^2(x,\xi)+r_0\right),$$
whence, by choosing $\delta_1$ sufficiently small and using hypotheses (H1) and (H3) (which is satisfied in $K_0'$), we have from \eqref{FP} that  there exists a positive constant $c$ such $a(x,\xi)\geq -c$ and hence $A$ satisfies the Fefferman-Phong inequality.
Finally, since $(A\varphi, \varphi)=(P'\varphi, \varphi)$ for all $\varphi \in C_0^{\infty}(K_0)$, we conclude that $P'$ satisfies the Fefferman-Phong inequality on $K_0$, that is, explicitly, there exists a positive constant $C$ such that, for all $\varphi \in C_0^\infty (K_0)$, $(P'\varphi,\varphi)\geq -C\n \varphi\n^2$.

Now, denoting by $K$ the compact containing $x_0$ in its interior where the Fefferman-Phong inequality for $P'$ holds, we have that \eqref{est2} is satisfied for all $\varphi\in C_0^\infty(K)$ (note that \eqref{est2} holds on each compact in $\Omega$) and we have from \eqref{ReP}
\begin{equation*}
2\mathsf{Re}(P^*\varphi, -i\underbrace{X_0}_{=X_0^*}\varphi)\geq (1-\delta_0-\delta_2)\n X_0 u\n^2 -\left[\frac{1}{\delta_2}\n d_{N+1}\n^2_{L^\infty(K)}+\frac {1}{\delta_0}\n\overline{a_0}\n^2_{L^\infty(K)}+\frac{1}{\delta_1}+C\right]  \n\varphi\n^2,
\end{equation*}
where $\delta_1$ is fixed here in order to have the Fefferman-Phong inequality for $P'$ in $K$ whereas $C$ is the related constant.\\
Since $2\mathsf{Re}(P^*\varphi, -iX_0^*\varphi)\leq \delta_3\n X_0\varphi\n^2+\frac{1}{\delta_3}\n P^*\varphi\n^2$, then we find
$$ \frac{1}{\delta_3}\n P^*\varphi\n^2 \geq (1-\delta_0-\delta_2-\delta_3)\n X_0 u\n^2 - C(K,\delta_0,\delta_2)\n \varphi\n^2.$$
We next choose $\delta_j$, $j=0,2,3$, sufficiently small so that $(1-\delta_0-\delta_2-\delta_3)\geq 1/2,$   and get
$$\n P^*\varphi\n^2\geq C_1 \n X_0\varphi\n^2-C_2 \n \varphi\n^2,$$
where the constants $C_1$ and $C_2$ are fixed now.
Finally, by applying a Poincar\'e inequality on $X_0$ (which is nondegenerate), and possibly by shrinking the compact $K$ around $x_0$ to a compact that we keep denoting by $K$, we have that there exists a positive constant $C$ such that, for all $\varphi\in C_0^\infty(K)$, one has
$$\n P^*\varphi\n^2\geq C \n \varphi\n^2,$$
which is the solvability estimate we were looking for. This concludes the proof.
\end{proof}

\begin{remark}\label{FinalRmk}
In the proof of Theorem \ref{Thm} we exploited conditions (H1), (H2) and (H3) in order to have that the symbol $p'(x,\xi)$ (to be precise its global extension $a(x,\xi)$) satisfies the hypothesis needed to apply the Fefferman-Phong inequality on $P'$.  The latter inequality applied in \eqref{ReP}, together with a Poincar\'e inequality on $X_0$ whose validity is granted by condition (H1) (i.e. $X_0(x,D)$ is nondegenerate in $\Omega$), gives the control of the second order quantity $(P'\varphi,\varphi)$ and the cancellation of $L^2$-errors. This means that conditions (H1), (H2) and (H3) are sufficient to get the solvability result for $P$ in $\Omega$. However, these hypothesis, are not even necessary for the $L^2$ to $L^2$ local solvability to hold. In fact, if no conditions are imposed on the vector fields involved in the expression of $P$, then $a(x,\xi)$ is of the form
\begin{equation}\label{FP1}
\begin{gathered}
a(x,\xi)=\chi_0^2\Big(\sum_{j=1}^N\left( \{p_0,f_j\}(p_j^1)^2+2f_j\{p_0,p_j^1\}p_j^1\right)+p_0^2-\delta_1q^2+r_0\Big)\\
=\chi_0^2\left(\sum_{j=1}^N \{X_0,f_j\}X_j^2+2f_j\{X_0,X_j\}+X_0^2-\delta_1\{X_0,X_{N+1}\}^2+r_0\right).
\end{gathered}
\end{equation}
Hence, if one can have the Fefferman-Phong inequality for $A$ (global extension of $P'$) being the Weyl quantization of the quantity above, then, still requiring the nondegeneracy of $X_0$, one can find the solvability result of Theorem \ref{Thm} by the same previous technique.

Note  finally that, when stronger conditions are satisfied, that is, when $P'$ is such that it satisfies the G\aa rding, the Melin or the Rothschild-Stein sharp subelliptic inequality, then one can get (starting from \eqref{ReP}) a better local solvability result, meaning that one can have $H^s$ to $H^{s'}$ local solvability with $(s,s')=(-1,0)$, $(s,s')=(-1/2,0)$ and $(s,s')=(-1/r,0)$, with $r>1$ integer, respectively. In fact we have that there exists a compact $K\subset\Omega$ containing $x_0$ in its interior and a positive constant $C$ such that $(P'\varphi, \varphi)\geq C \n \varphi \n_{-s}$, with $-s=1,1/2, 1/r$, if $P'$ satisfies the G\aa rding, the Melin or the Rothschild-Stein inequality respectively. We then have from \eqref{ReP}

$$ \delta_3\n X_0\varphi\n^2+\frac{1}{\delta_3}\n P^*\varphi\n^2\geq 2\mathsf{Re}(P^*\varphi, -iX_0\varphi)\geq C\n \varphi\n^2_{-s}+(1-\delta_0-\delta_2)\n X_0\varphi\n^2$$
$$-\left[\frac{1}{\delta_2}\n d_{N+1}\n^2_{L^\infty(K)}+\frac {1}{\delta_0}\n\overline{a_0}\n^2_{L^\infty(K)}+\frac{1}{\delta_1}\right] \n\varphi\n^2,$$
hence, as before, by suitably choosing the constants $\delta_j$, $j=0,1, 2,3,$ and absorbing the $L^2$-errors with the $-s$ norm one obtains $H^{s}$ to $H^0$ local solvability.
\end{remark}

\section{EXAMPLES OF OPERATORS IN THE CLASS}
In this section we shall give some examples of operators in the class \eqref{P}.
\begin{example}
Consider the operator defined in $\mathbb{R}^n$, $n\geq 2$, of the form
\begin{gather*}P=x_1(D_1^2-D_2^2)+i(D_1+D_2)+X(x,D)\\
=D_1x_1D_1-D_2x_1D_2+iD_2+X(x,D),
\end{gather*}
where $X(x,D)$ is a first order homogenous partial differential operator with real smooth coefficients of the form
$$X(x,D)=g_1(x_1)D_1+g_2(x_1,x_2)D_2$$
when $n=2$, with $g_1\in C^\infty (\mathbb{R})$, $g_2\in C^\infty (\mathbb{R}^2)$,  and of the form
$$X(x,D)=g_1(x_1,x_3,..,x_n)D_1+g_2(x)D_2+\sum_{j=3}^ng_j(x_1,x_3,...,x_n)D_j,$$
when $n\geq3$, with $g_2\in C^\infty (\mathbb{R}^n)$ and $g_j\in C^\infty(\mathbb{R}^{n-1};\mathbb{R})$.\\
It is easy to check that conditions (H1), (H2) and (H3) are satisfied by $P$, hence, by the theorem above, we have $L^2$ to $L^2$ local solvability for $P$ at each point of $\mathbb{R}^n$, $n\geq 2$. 
\end{example}

\begin{example}
Consider the operator defined in $\mathbb{R}^{n+1}$ of the form
$$P=\sum_{j=1}^k D_j x_j^p D_j\pm\sum_{j'=k+1}^nD_{j'} x_{j'}^{p'} D_{j'}+if(t)D_t+\sum_{\ell=1}^ng_\ell(x)D_\ell,$$
where $D_j=D_{x_j}$, $k\leq n$ is a positive integer, $p,p'\in \mathbb{N}$ and $g_\ell$ and $f$  are real smooth functions with $f$ nonvanishing. Again one has that conditions (H1), (H2) and (H3) are satisfied, therefore $P$ is $L^2$ to $L^2$ locally solvable in $\mathbb{R}^{n+1}$.
\end{example}

\begin{example}
Let $n+1\geq 3$ and 
$$X_1(x,t,D_x,D_t)=g(x_1)D_1+D_2;\quad X_2(x,t,D_x,D_t)=D_t+h(x)D_2; $$
$$X_3(x,t,D_x, D_t)=\sum_{j=1}^n k_j(x)D_{j}+k_{n+1}(x,t)D_t;\quad X_0(x,t,D_x, D_t)=D_t,$$
where $D_j=D_{x_j}$, $g\in C^\infty(\mathbb{R};\mathbb{R})$, $h\in C^\infty(\mathbb{R}^n;\mathbb{R}), k_i \in C^\infty(\mathbb{R}^n;\mathbb{R}),$ for all $i=1,...,n,$ and $k_{n+1}\in C^\infty(\mathbb{R}^{n+1};\mathbb{R})$. Let $P$ be
$$P= X_1^* t^{2p+1}X_1+X_2^* t^{2p'+1}X_2+iX_0+X_3,$$
with $p$ and $p'$ positive integers. Then, again, the hypotheses of Theorem \ref{Thm} apply and we get $L^2$ to $L^2$ local solvability for $P$ at each point of $\mathbb{R}^{n+1}$. Note that in this example the vector fields $X_1$ and $X_2$ do not form an involutive distribution. This shows that the class \eqref{P} generalizes the class of mixed type operators in \cite{FP1} where the presence of functions $f_j=f$, for all $1\leq j\leq N$, is required and a strict sign condition of the form $iX_0(x,D)f>0$ on $f^{-1}(0)$ is needed. Moreover, with respect to the class of Shr\"{o}dinger type operators in \cite{FP1} where, instead, the presence of several functions $f_j$ is allowed and $X_0, X_{N+1}$ are assumed to be such that $X_0\equiv 0$ and $X_{N+1}\not\equiv 0$, here we do not require any involutive structure of the vector fields $X_j$, $1\leq j\leq N$, whereas in \cite{FP1} an involutivity property is considered.
\end{example}

\begin{example}
This last example is to show that there are cases in which we can have a better kind of local solvability for operators in the class $\eqref{P}$.

Consider in $\mathbb{R}^2$ the operator
$$P=D_1x_1D_1-D_2x_2D_2+i(D_1-D_2)+x_2D_1$$
$$=x_1D_1^2-x_2D_2^2+x_2D_1,$$
where 
$$X_1(x,D)=D_1,\,\, X_2(x,D)=D_2,\,\, X_0(x,D)=D_1-D_2,\,\, X_3(x,D)=x_2D_1,$$
$$ f_1(x)=x_1,\,\,f_2(x)=-x_2,$$
and note that conditions (H1), (H2) and (H3) are satisfied. Moreover $P'$ is such that its Weyl symbol is given by (see \eqref{P'} and \eqref{FP1})
$$p'(x,\xi)=\xi_1^2+\xi_2^2+(\xi_1-\xi_2)^2-\delta_1\xi_1^2,$$
whence (by choosing $\delta_1$ sufficiently small) $P'$ satisfies the G\aa rding inequality and by Remark \ref{FinalRmk} we have that $P$ is  $H^{-1}$ to $L^2$ locally solvable at each point of $\mathbb{R}^2$.

To conclude we just want to say that the same result as before holds for the operator
$$P=D_12x_1D_1-D_22x_2D_2+i(D_1-D_2)+x_2D_1$$
$$=x_1D_1^2-x_2D_2^2-i(D_1-D_2)+x_2D_1,$$
whose symbol $p'$ is
$$p'(x,\xi)=2\xi_1^2+2\xi_2^2+(\xi_1-\xi_2)^2-\delta_1\xi_1^2.$$
\end{example}


\end{document}